\documentclass
{amsart}

\usepackage[latin1]{inputenc}

\usepackage{amsfonts}
\usepackage{latexsym}
\usepackage{color}
\theoremstyle{plain}
\newtheorem{theorem}[subsection]{Theorem}
\newtheorem{definition}[subsection]{Definition}
\newtheorem{Pro}[subsection]{Proposition}
\newtheorem{remark}[subsection]{Remark}
\newtheorem{Cor}[subsection]{Corollary}

\newtheorem{lemma}[subsection]{Lemma}

\numberwithin{equation}{subsection}
\begin{document}
\author[X. Dom\'{\i}nguez,  V. Tarieladze]{\sc X. Dom\'{\i}nguez, V. Tarieladze}
 \footnotetext{
 Partially supported by MEC of Spain BFM2006- 
03036 and Gobierno de Navarra; the second named author was supported also by  GNSF/ST08/3-384. }

\title{Metrizable  TAP, HTAP and  STAP groups}

\begin{abstract}
In a recent paper by D. Shakhmatov and  J. Sp\v ev\'ak [Group-valued continuous functions with the topology of pointwise convergence, Topology 
and its Applications (2009), doi:10.1016/j.topol.2009.06.022] the concept of a ${\rm TAP}$ group is introduced and it is shown in particular that ${\rm NSS}$ groups are ${\rm TAP}$. We prove that conversely, Weil complete metrizable ${\rm TAP}$ groups are ${\rm NSS}$. We define also  the narrower class of ${\rm STAP}$ groups,  show that the ${\rm NSS}$ groups are in fact ${\rm STAP}$ and that the converse statement is true in metrizable case. A remarkable characterization of pseudocompact spaces obtained in  the paper by D. Shakhmatov and  J. Sp\v ev\'ak asserts: a Tychonoff space $X$ is pseudocompact if and only if $C_p(X,\mathbb R)$ has the  ${\rm TAP}$ property. We show that for no infinite  Tychonoff space $X$, the group $C_p(X,\mathbb R)$ has the  ${\rm STAP}$ property. We also show that a metrizable locally balanced  topological vector group is ${\rm STAP}$ iff it does not contain a subgroup topologically isomorphic to $\mathbb Z^{(\mathbb N)}$.
\end{abstract}
\maketitle
\section{Introduction}
In \cite[Definition 4.1]{SS} a subset $A$ of of a topological group $G$ is called {\it absolutely productive} in $G$ provided that, for every injection $a:\mathbb N\to A$ and each mapping $z:\mathbb N\to \mathbb Z$ the sequence
 $$
 \left( \prod_{k=1}^n a(k)^{z(k)} \right)_{n\in \mathbb N}
 $$
converges to  some $g\in G$.
\par
In \cite[Definition 4.5]{SS} a topological group $G$ is called ${\rm TAP}$ (an abbreviation for ``trivially absolutely productive") if every absolutely productive set in $G$ is finite.
\par
For a topological group $G$ we call a (not necessarily injective) $a:\mathbb N\to G$
\begin{itemize}
\item[(h-m)]
 {\it hyper-multipliable} in $G$ provided that for each mapping $z:\mathbb N\to \mathbb Z$ the sequence
 $
 \left( \prod_{k=1}^n a(k)^{z(k)} \right)_{n\in \mathbb N}
 $
converges to  some $g\in G$.
\item[(h-c)]
{\it hyper-converging} in $G$ provided that for each mapping $z:\mathbb N\to \mathbb Z$ the sequence $
 \left( a(n)^{z(n)} \right)_{n\in \mathbb N}
 $
converges to  the neutral element $e$ of  $G$.
\end{itemize}

We call a topological group $G$ 
\begin{itemize}
\item[(I)]  ${\rm HTAP}$ (an abbreviation for ``Hyper ${\rm TAP}$") iff no injective $a:\mathbb N\to G$ is  hyper-multipliable in $G$.
\item[(II)]  ${\rm STAP}$ (an abbreviation for ``Strictly ${\rm TAP}$") iff no injective $a:\mathbb N\to G$ is hyper-converging  in $G$.
\end{itemize}
\par
It is clear that if a topological group $G$ is ${\rm STAP}$ then it is ${\rm HTAP}$ as well, while the converse statement may fail even for a precompact countable metrizable abelian group (see Lemma \ref{co3}$(b)$).
\par
It is clear also that if  $G$ is ${\rm HTAP}$ then it is ${\rm TAP}$; the converse is true provided $G$ is {\it abelian} (see Lemma \ref{co3}). We do not know whether a non-abelian  ${\rm TAP}$ group is necessarily ${\rm HTAP}$.  
 
\par
In \cite[Theorem 4.9]{SS} it is proved  that  ${\rm NSS}$ groups are ${\rm TAP}$ and that the converse statement fails in general. We show that ${\rm NSS}$ groups are in fact ${\rm STAP}$ and that the converse statement is true in metrizable case (Theorem \ref{shakh2}). We prove also that for Weil-complete metrizable groups ${\rm NSS}$ and ${\rm TAP}$ properties are equivalent (Theorem \ref{shakh3}). A proof of this result contained in a previous version of this paper (see \cite {tapstap}) is not complete (see Remark \ref{ars-sh}).
\par
A remarkable characterization of pseudocompact spaces obtained in  \cite[Theorem 4.9]{SS} says: a Tychonoff space $X$ is pseudocompact if and only if $C_p(X,\mathbb R)$ has the  ${\rm TAP}$ property. We show that for no  infinite  Tychonoff space $X$, the group $C_p(X,\mathbb R)$ has the  ${\rm STAP}$ property (Theorem \ref{pseudo2m}).
\par
In Section \ref{hyp-hyp} the hyper-multipliable, hyper-converging  and related sequences are discussed. 
\par
Section \ref{stap} deals with general ${\rm TAP}$,  ${\rm STAP}$ and ${\rm NSS}$ groups.
\par
In Section \ref{cepe} the case of $C_p(X)$ is considered.
\par
In   Section \ref{tvp}, along the same lines as \cite{BPR}, we characterize the (complete) metrizable topological vector groups over $\mathbb R$ which are 
({\rm TAP}$)  \,\,  {\rm STAP}$. \\
\\
By $\mathbb N$ we denote the set $\{1,2,\dots\}$ of natural numbers.
\par
As in \cite{SS} $G$ will stand for a Hausdorff topological group with internal operation $\cdot$ and with the neutral element $e$. 
In case of abelian groups with the internal operation $+$, the neutral element will be denoted by $0$.
\par
We denote by $\mathcal N(G)$ the set of all neighborhoods of $e$ in $G$.
\par
We denote by $\langle g\rangle$  the cyclic subgroup of $G$ generated by $g$, for an element $g\in G$.
\par
The group $G$ is a  {\rm{NSS} group if $G$ has an open neighborhood of $e$ containing no nontrivial subgroups of $G$.
\par
For a nonempty family $(G_i)_{i\in I}$ of topological groups the product group $\prod_{i\in I}G_i$ is endowed with the product topology. We write:
$$
{\rm Supp}(g)=\{i\in I:g(i)\ne e_{G_i}\,\},\quad g\in \prod_{i\in \mathbb N}G_i
$$
and
$$
\Pi_{i\in I}^{({\rm r})}G_i=\{g\in \prod_{i\in \mathbb N}G_i: {\rm Card}({\rm Supp}(g))<\infty\}.
$$ 
When $G_i=G$ for every $i\in I$,
$$G^I:=\prod_{i\in I}G_i,\quad G^{(I)}:=\Pi_{i\in I}^{({\rm r})}G_i\,.$$
 
\section{Hyper-multipliable and related sequences.}\label{hyp-hyp}
We call a sequence $(g_n)_{n\in \mathbb N}$ extracted from $G$:
\begin{itemize}
\item
{\it eventually neutral} if the set $\{n \in \mathbb N:g_n\ne e\}$ is {\em finite}.
\item
 {\it hyper-multipliable} in $G$ 
  if for {\em every} sequence $(m_n)_{n\in \mathbb N}$ extracted from $\mathbb Z$ the sequence
 $
 \left( \Pi_{k=1}^n g_k^{m_k} \right)_{n\in \mathbb N}
 $
converges to  some $g\in G$.
\item
 {\it super-multipliable} (or {\it hereditarily multipliable}) in $G$  if for {\em every} sequence $(m_n)_{n\in \mathbb N}$ extracted from $\{0,1\}\subset \mathbb Z$ the sequence
 $
 \left( \Pi_{k=1}^n g_k^{m_k} \right)_{n\in \mathbb N}
 $
 converges to  some $g\in G$.
\item
 {\it hyper-converging} if for {\em every} sequence $(m_n)_{n\in \mathbb N}$ extracted from $\mathbb Z$ the sequence
 $
 \left( g_n^{m_n} \right)_{n\in \mathbb N}
 $
converges to $e$.
\end{itemize}
In case of abelian groups with the internal operation $+$, instead of the terms ``eventually neutral", ``hyper-multipliable", ``super-multipliable",   
we shall use the terms: ``eventually zero", ``hyper-summable", ``super-summable".

\begin{lemma}\label{co1}
Let $(g_n)_{n\in \mathbb N}$ be a sequence extracted from $G$.
\begin{itemize}
\item[(a)]  $(g_n)_{n\in \mathbb N}$ is eventually neutral \,\,$\Longrightarrow$\,\,  $(g_n)_{n\in \mathbb N}$ is hyper-multipliable.
\\ The converse implication fails  e.~g.~
when $G=\prod_{i\in \mathbb N}G_i$, where $(G_i)_{i\in \mathbb N }$ is a sequence of nontrivial Hausdorff topological groups {\em (cf. \cite [Lemma 4.4]{SS})}.
\item[(b)] $(g_n)_{n\in \mathbb N}$ is hyper-multipliable \,\,$\Longrightarrow$\,\,  $(g_n)_{n\in \mathbb N}$ is hyper-converging.\\ 
The converse implication fails  e.g.,
when $H=\prod_{i\in \mathbb N}^{({\rm r})}G_i$ with induced from $\prod_{i\in \mathbb N}G_i$ topology, where $(G_i)_{i\in \mathbb N }$ is a sequence of nontrivial Hausdorff topological groups . 
\\
In particular, 
the converse implication may fail for a countable precompact metrizable abelian group.
\item[(c)]  $(g_n)_{n\in \mathbb N}$ is hyper-multipliable \,\,$\Longrightarrow$\,\, $(g_n)_{n\in \mathbb N}$ is super-multipliable \,\,$\Longrightarrow$ $(g_n)_{n\in \mathbb N}$ is convergent to $e$. The converse implications are not true.

\end{itemize}
\end{lemma}
\begin{proof}
\begin{itemize}
\item[(a)] The first part is evident.\\
 Clearly $e:\mathbb N\to \cup_{i\in \mathbb N}G_i$ with $e(i)=e_{G_i},\,i=1,2,\dots$ is the neutral element of $G=\prod_{i\in \mathbb N}G_i$. Now for each $i\in \mathbb N$ choose some $x_i\in G_i\setminus \{e_{G_i}\}$ and define a sequence  $(g_n)_{n\in \mathbb N}$ as follows: for a fixed  $n\in \mathbb N$, set $g_n(n)=x_n$ and $g_n(i)=e_{G_i}$ when $i\in \mathbb N\setminus \{n\}$. Clearly thus obtained sequence $(g_n)_{n\in \mathbb N}$ is not eventually neutral in $G$. Let us see that $(g_n)_{n\in \mathbb N}$ is hyper-multipliable in $G$. In fact, extract from $\mathbb Z$ a sequence $(m_n)_{n\in \mathbb N}$ and consider $g\in G$ with $g(i)=x_i^{m_i},\,\,i=1,2,\dots$ Then the sequence $ \left( \Pi_{k=1}^n g_k^{m_k} \right)_{n\in \mathbb N}$
 converges in $G$ to $g$.
\item[(b)] The first part is evident.\\
    Let $G$ and the sequence $(g_n)_{n\in \mathbb N}$ be as in the proof of (a). 
     Then $g_n\in H,\,n=1,2,\dots$; $(g_n)_{n\in \mathbb N}$ hyper-converges to $e$ in $H$,  but $(g_n)_{n\in \mathbb N}$ is not hyper-multipliable in $H$.\\
If the groups $G_i,\,i=1,2,\dots$ are finite discrete  abelian, then $\prod_{i\in \mathbb N}G_i$ is compact metrizable abelian and hence, $H$ is countable precompact metrizable abelian.
\item[(c)] The proof is left to the reader.
\end{itemize}
\end{proof}
\begin{remark}\label{meore}{\em In connection with  Lemma \ref{co1}$(b)$ we note that in  a sequentially complete abelian locally quasi-convex $G$ all hyper-converging to $e$ sequences are hyper-multipliable (see \cite{dtslender}, where it is shown also that a similar statement may not be true for a complete metrizable abelian non locally quasi-convex $G$). 

}
\end{remark}
\begin{lemma}\label{vigo1}
Let $(g_n)_{n\in \mathbb N}$ be a sequence extracted from $G$.
\begin{itemize}
\item[(a)] If $(g_n)_{n\in \mathbb N}$ is hyper-multipliable in $G$ and $(r_n)_{n\in \mathbb N}$ is {\it a strictly increasing sequence of natural numbers}, then the sequence $(g_{r_n})_{n\in \mathbb N}$ is hyper-multipliable in $G$.
\item[(b)] If $(g_n)_{n\in \mathbb N}$ is super-multipliable in $G$ and $(r_n)_{n\in \mathbb N}$ is {\it a strictly increasing sequence of natural numbers}, then the sequence $(g_{r_n})_{n\in \mathbb N}$ is super-multipliable in $G$.
\end{itemize}
\end{lemma}
\begin{proof}
$(a)$ In fact, fix a sequence $m:\mathbb N\to \mathbb Z$ and define $\tilde{m}:\mathbb N\to \mathbb Z$ 
as follows:
$\tilde{m}_{r_n}=m_n,\,n=1,2,\dots$ and $\tilde{m}_j=0,\forall j\in \mathbb N\setminus\{r_1,r_2,\dots\}$. Then we will have:
$$
\Pi_{k=1}^{r_n} g_k^{\tilde{m}_k}=\Pi_{k=1}^{n} g_{r_k}^{m_k},\,\,n=1,2,\dots
$$
Since $(g_n)_{n\in \mathbb N}$ is hyper-multipliable in $G$, the sequence $\left( \Pi_{k=1}^{r_n} g_k^{\tilde{m}_k}\right)_{n\in \mathbb N}$ converges to some $g\in G$. From this and the above equality we get that the sequence  $\left(\Pi_{k=1}^{n} g_{r_k}^{m_k} \right)_{n\in \mathbb N}$ converges to $g\in G$ as well. Since $m:\mathbb N\to \mathbb Z$ is arbitrary, we proved that the sequence $(g_{r_n})_{n\in \mathbb N}$ is hyper-multipliable in $G$.
\par
$(b)$ can be proved in a similar way.
\end{proof}
\begin{lemma}\label{vigo2}
Let $G$ be {\em abelian} and $(g_n)_{n\in \mathbb N}$ be a sequence extracted from $G$.
\begin{itemize}
\item[(a)] If  $(g_n)_{n\in \mathbb N}$ is super-multipliable in $G$ and $\pi:\mathbb N\to \mathbb N$ is a bijection, then   the sequence 
 $$
 \left( \Pi_{k=1}^n g_{\pi(k)} \right)_{n\in \mathbb N}
 $$
 converges to  some $g_{\pi}\in G$ (moreover, then $g_{\pi}$ does not depend on $\pi$).

\item[(b)] If  $(g_n)_{n\in \mathbb N}$ is hyper-multipliable in $G$ and $(k_n)_{n\in \mathbb N}$ is {\it an injective  sequence of natural numbers}, then the sequence $(g_{k_n})_{n\in \mathbb N}$ is hyper-multipliable in $G$.

\end{itemize}

\end{lemma}
\begin{proof}
$(a)$ will be proved in the appendix.\\
$(b)$ Fix a sequence $m:\mathbb N\to \mathbb Z$, find a bijection $\sigma:\mathbb N\to \mathbb N$ such that the sequence $r_n:=k_{\sigma(n)},\,n=1,2,\dots$ is strictly increasing and define   $m':\mathbb N\to \mathbb Z$ 
as follows: $m'_n=m_{\sigma(n)},\,n=1,2,\dots$. Then, by Lemma \ref{vigo1}$(a)$ the sequence $(g_{r_n})_{n\in \mathbb N}$ is hyper-multipliable in $G$. This implies that the sequence $(g_{r_n}^{m'_n})_{n\in \mathbb N}$ is hyper-multipliable in $G$ too. In particular, the sequence $(g_{r_n}^{m'_n})_{n\in \mathbb N}$ is super-multipliable in $G$.  Write: $h_n:= g_{r_n}^{m'_n},\,n=1,2,\dots$ Since the sequence $(h_n)_{n\in\mathbb N}$ is super-multipliable in $G$, according to $(a)$, the sequence 
$$
\left( \Pi_{j=1}^n h_{\sigma^{-1}(j)} \right)_{n\in \mathbb N}
$$
converges to some $h\in G$. Since
$$
h_{\sigma^{-1}(j)}=g_{k_j}^{m_j},\,j=1,2,\dots\,,
$$
we obtained that the sequence
$$
\left( \Pi_{j=1}^n  g_{k_j}^{m_j}\right)_{n\in \mathbb N}
$$
converges to  $h\in G$. Since $m:\mathbb N\to \mathbb Z$ is arbitrary, we proved that the sequence $(g_{k_n})_{n\in \mathbb N}$ is hyper-multipliable in $G$.
\end{proof}
\begin{remark}\label{preg}{\em We do not know whether Lemma \ref{vigo2}$(a)$ (of course without its `moreover' part) remains true for a non-abelian $G$ as well.
}
\end{remark}

\begin{Pro}\label{agv14}
Let $(D_i)_{i\in \mathbb N}$ be a sequence of  non-trivial discrete groups. Put $G=\prod_{i\in \mathbb N}D_i$ and let $H$ be $\prod_{i\in \mathbb N}^{({\rm r})}D_i$ with the topology induced by $G$.
 Then the complete metrizable non-discrete  topological group $G$  and the non-complete  metrizable non-discrete  topological group $H$ have the following properties: 
\begin{itemize}
\item[(1)] {\em Every}
sequence of elements of $G$ which converges to $e$ is hyper-multipliable in $G$.
In particular, {\em every} sequence of elements of $G$ which converges to $e$ is hyper-converging.
\item[(2)] If a sequence  $(h_n)_{n\in \mathbb N}$ of elements of $H$ is super-multipliable in $H$, then $(h_n)_{n\in \mathbb N}$ is a eventually neutral  sequence.
\item[(3)] If a sequence  $(h_n)_{n\in \mathbb N}$ of elements of $H$ is hyper-multipliable  in $H$, then $(h_n)_{n\in \mathbb N}$ is a eventually neutral  sequence.
\end{itemize}
\end{Pro}
\begin{proof}\begin{itemize}
\item[(1)] Observe that, since $D_{i},\,i=1,2,\dots$ are discrete, the open subgroups
$$
U_m=\{g\in G: g(i)=e_{D_i},\,i=1,\dots,m\},\,\,m=1,2,\dots
$$
form a basis for $\mathcal N(G)$. From this observation and from the sequential completeness of $G$ it is standard to deduce that (1) is true.\item[(2)] Fix a not eventually neutral  sequence $(h_n)_{n\in \mathbb N}$ of elements of $H$, which converges to $e$ in $H$ and let us see that 
 $(h_n)_{n\in \mathbb N}$ is not super-multipliable in $H$.
 
 We can assume without loss of generality that 
 \begin{equation}\label{15.0}
 {\rm Supp}(h_n)\ne \emptyset,\,\,n=1,2,\dots
 \end{equation}
 A little reflection shows that since $(h_n)_{n\in \mathbb N}$ tends to $e$ in $H$ {\it and the groups $D_i$'s are discrete}, we have
 \begin{equation}\label{15.1}
 \min\left( {\rm Supp}(h_n) \right)\longrightarrow \infty\,\,(n\to \infty)\,.
  \end{equation}
  Construct now a strictly increasing sequence $(q_n)_{n\in \mathbb N}$ of natural numbers as follows.\\
  Take $q_1=1$. From (\ref{15.1}) we can find $q_2\in \mathbb N$ such that $q_2>q_1$ and
  $$
  \min\left( {\rm Supp}(h_{q_2})\right)>\max\left( {\rm Supp}(h_{q_1})\right)
  $$
  Continuing in this manner, we obtain a strictly increasing sequence $(q_n)_{n\in \mathbb N}$ of natural numbers such that
  $$
  \min\left( {\rm Supp}(h_{q_{n+1}})\right)>\max\left( {\rm Supp}(h_{q_n})\right),\,\,n=1,2,\dots
  $$
  From the last relation we get
  \begin{equation}\label{15.2}
  n',n''\in \mathbb N,\,n'\ne n''\,\Longrightarrow\,{\rm Supp}(h_{q_{n'}})\cap {\rm Supp}(h_{q_{n''}})=\emptyset\,.
  \end{equation}
  Now, since the sequence $(h_{q_n})_{n\in \mathbb N}$ tends to $e$ in $H$, it tends to $e$  in $G$ too. Then, according to (1) there is $g\in G$ such that
  $$
  \lim_n\prod_{k=1}^nh_{q_k}=g\,
  $$
  From (\ref{15.2}) it follows that
  \begin{equation}\label{15.3}
    {\rm Supp}(h_{q_n})\subset {\rm Supp}(g),\,\,n=1,2,\dots
  \end{equation}
    From (\ref{15.2}) and (\ref{15.0}) we conclude that $ {\rm Supp}(g)$ {\em is not} a finite set. Consequently, $g\not\in H$ and $(h_n)_{n\in \mathbb N}$ is not super-multipliable in $H$.\\
   \item[(3)] follows from (2) by the first implication of Lemma \ref{co1}$(c)$.   
   \end{itemize}
\end{proof}

\section{Absolutely productive sets}
In terms of hyper-multipliable sequences the definition of an absolutely productive set can be formulated as follows.

\begin{definition}\label{abspro} {\em (cf.~\cite[Definition 4.1]{SS}) A subset $A$ of $G$ is {\it absolutely productive in} $G$ if every sequence $(g_n)_{n\in \mathbb N}$ {\em of pairwise distinct elements} extracted from $A$ is hyper-multipliable in $G$.
}
\end{definition}

\begin{lemma}\label{amoc} Let $(g_n)_{n\in \mathbb N}$ be a sequence of pairwise distinct elements of $G$ and $A:=\{g_1,g_2,\dots\}$. Consider the statements:
\begin{itemize}
\item[(ap1)] $A$ is absolutely productive in $G$.
\item[(ap2)] For every injection $\sigma:\mathbb N\to \mathbb N$ the sequence $(g_{\sigma(n)})_{n\in \mathbb N}$ is hyper-multipliable in $G$.
\item[(ap3)] For every bijection $\sigma:\mathbb N\to \mathbb N$ the sequence $(g_{\sigma(n)})_{n\in \mathbb N}$ is hyper-multipliable in $G$.
\item[(ap4)] $(g_n)_{n\in \mathbb N}$ is hyper-multipliable in $G$.
\end{itemize}
Then the following statements are valid:
\begin{itemize}
\item[(a)]
 (ap1), (ap2) and (ap3) {\it are equivalent and they imply} (ap4).
\item[(b)]  {\em If $G$ is abelian}, then (ap1), (ap2), (ap3)  and (ap4)  are equivalent.
 \end{itemize}
 
\end{lemma}
\begin{proof}
$(a)$\\
The equivalence (ap1)$\Longleftrightarrow$ (ap2) follows directly from the definitions.\\
The implication (ap2)$\Longrightarrow$ (ap3) is evident.\\ 
(ap3)$\Longrightarrow$ (ap1).\\
Consider an injection $a:\mathbb N\to A$,  a mapping $z:\mathbb N\to \mathbb Z$ and let us see that  the sequence
\begin{equation}\label{bolo2}
 \left( \Pi_{k=1}^n a(k)^{z(k)} \right)_{n\in \mathbb N}
 \end{equation}
converges to  some $g\in G$.
\par
For every $n\in \mathbb N$ we can find and fix some $k_n\in \mathbb N$ such that $a(n)=g_{k_n}$. Since $a$ is injective, thus obtained sequence $(k_n)_{n\in \mathbb N}$ is injective as well.\\
Find a bijection $\sigma:\mathbb N\to \mathbb N$ such that the sequence $r_n:=k_{\sigma(n)},\,n=1,2,\dots$ is strictly increasing.\\
Define now the mappings $\tilde \sigma:\mathbb N\to \mathbb N$ and $\tilde m:\mathbb N\to \mathbb Z$ as follows:\\
$\tilde\sigma(r_n)=k_n,\,n=1,2,\dots,\,\,\,\tilde \sigma(j)=j,\,\forall j\in \mathbb N\setminus \{r_1,r_2,\dots\}$;\\
$\tilde m_{r_n}=z(n),\,n=1,2,\dots,\,\,\,\tilde m_j=0,\,\forall j\in\mathbb N\setminus \{r_1,r_2,\dots\}$.
\par
Clearly, $\tilde \sigma$ is a bijection. Hence, since (ap3) is satisfied, we have that the sequence 
$$
\left( \prod_{i=1}^ng_{\tilde\sigma(i)} ^{\tilde m_i}\right)_{n\in \mathbb N}
$$
converges to some $g\in G$. Hence,
$$
\lim_n \left( \prod_{i=1}^{r_n}g_{\tilde\sigma(i)} ^{\tilde m_i}\right)=g\,.
$$
Observe that
$$
\prod_{i=1}^{r_n}g_{\tilde\sigma(i)} ^{\tilde m_i}=\prod_{k=1}^n a(k)^{z(k)},\,\,n=1,2,\dots
$$
The last two equalities imply that
$$
\lim_n\left(\prod_{k=1}^n a(k)^{z(k)}\right)=\lim_n \left( \prod_{i=1}^{r_n}g_{\tilde\sigma(i)} ^{\tilde m_i}\right)=g\,.
$$
Therefore, the sequence (\ref{bolo2}) converges to $g$.
 Since the injection $a:\mathbb N\to A$ is arbitrary and the mapping   $z:\mathbb N\to \mathbb Z$ is arbitrary too, we proved that $A$ is absolutely productive in $G$.\\
The implication (ap3)$\Longrightarrow$ (ap4) is evident.
\par
$(b)$ Since $(a)$ is proved, we need to prove only that (ap4) implies (ap2). Since $G$ is abelian, this implication is true by Lemma \ref{vigo2}$(b)$.
\end{proof}

\section{${\rm TAP}$ and ${\rm STAP}$ groups: the first observations}\label{stap}

Note that, according to Definition \ref{abspro}, every {\it finite} subset $A$ of $G$ is { absolutely productive in} $G$.
\begin{definition}{\em (\cite[Definition 4.5]{SS}) We say that  $G$ is $\rm{TAP}$ (an abbreviation for ``trivially absolutely productive") if every absolutely productive set in $G$ is finite.
 }
\end{definition}
Motivated from  the concept of a  $\rm{TAP}$ group, we propose the following definitions.

\begin{definition}{\em  We say that  $G$ is 
\begin{itemize}
\item[(I)]  ${\rm HTAP}$ (an abbreviation for ``Hyper ${\rm TAP}$") if no sequence {\em of pairwise distinct elements} extracted from $G$ is hyper-multipliable.
 \item[(II)]  ${\rm STAP}$ (an abbreviation for ``Strictly ${\rm TAP}$")  no sequence $(g_n)_{n\in \mathbb N}$ {\em of pairwise distinct elements} extracted from $G$ is hyper-convergent  in $G$.
\end{itemize}
}
\end{definition}
The following lemmas are easy to prove.
\begin{lemma}\label{co20}
For a topological group $G$ the following conditions are equivalent:
\begin{itemize}
\item[(i)] $G$ is ${\rm HTAP}$.
\item[(ii)] Every sequence extracted from $G$ which is hyper-multipliable  in  $G$ is eventually neutral.
\end{itemize}
\end{lemma}

\begin{lemma}\label{co2}
For a topological group $G$ the following conditions are equivalent:
\begin{itemize}
\item[(i)] $G$ is ${\rm STAP}$.
\item[(ii)] Every hyper-convergent  sequence extracted from $G$ is eventually neutral. 
\end{itemize}
\end{lemma}

The following lemma contains, in particular, some examples of ${\rm TAP}$, ${\rm HTAP}$ and ${\rm STAP}$ groups.

\begin{lemma}\label{co3}
We have:
\begin{itemize}
\item[(a)] $G\in {\rm STAP}\Longrightarrow G\in {\rm HTAP}\Longrightarrow G\in {\rm TAP}$.
\item[(a-bis)] If $G$ is {\em abelian} then $G\in {\rm HTAP}\Longleftrightarrow G\in {\rm TAP}$.
\item[(b)] Let $(D_i)_{i\in \mathbb N}$ be a sequence of  non-trivial discrete groups and   $H$ be $\prod_{i\in \mathbb N}^{({\rm r})}D_i$ with the topology induced by $\prod_{i\in \mathbb N}D_i$. Then  $H\in {\rm HTAP}$, but  $H\not\in {\rm STAP}$.\\
In particular, the implication $H\in {\rm TAP}\Longrightarrow H\in {\rm STAP}$ may fail even for an abelian countable metrizable precompact  $H$.
\end{itemize}
\end{lemma}
\begin{proof}
(a). The implication $G\in {\rm STAP}\Longrightarrow G\in {\rm HTAP}$ is true in view of Lemma \ref{co1}$(b)$.\\
The implication $G\in {\rm HTAP}\Longrightarrow G\in {\rm TAP}$ follows directly from the definitions.
\\
(a-bis) is true because of Lemma \ref{amoc}$(b)$.\\
(b) $H\in {\rm HTAP}$ by Proposition \ref{agv14}(3). $H\not\in {\rm STAP}$ because, since $H$ is metrizable and non-discrete, we can extract from $H$ a not eventually neutral sequence which converges to $e$ and  by Proposition \ref{agv14}(1) every such sequence hyper-converges to $e$.
\par
It follows that if the groups $D_i,i=1,2,\dots$ are finite and abelian, then $H$ is an example  of an  abelian countable metrizable precompact group, for which 
the implication $H\in {\rm TAP}\Longrightarrow H\in {\rm STAP}$ fails.
\end{proof}

We shall see below that the ${\rm TAP}$ and the ${\rm STAP}$ properties are equivalent for a {\it metrizable} Weil complete group (see Theorem \ref{shakh3}).

\section{  ${\rm NSS}$, ${\rm STAP}$ and ${\rm TAP}$ groups}\label{stap2}
\begin{theorem}\label{shakh}
We have: 
\begin{itemize}
\item[(a)] $G\in {\rm NSS}\Longrightarrow G\in {\rm STAP}$ {\em (see Theorem \ref{shakh2} for a converse)}.
\item[(b)] {\em (\cite[Theorem 4.9]{SS})} $G\in {\rm NSS}\Longrightarrow G\in {\rm TAP}$.
\item[(c)] In general, $G\in {\rm HTAP}\not\Longrightarrow G\in {\rm NSS}$ even for an abelian countable metrizable precompact $G$.
\end{itemize}
\end{theorem}
\begin{proof}
\begin{itemize}
\item[(a)] Fix a sequence  $(g_n)_{n\in \mathbb N}$ {\em of pairwise distinct elements} extracted from $G$. By Lemma \ref{co2} (ii)$\Longrightarrow$(i) it is sufficient to show that $(g_n)_{n\in \mathbb N}$ is not hyper-convergent. 
 We can assume without loss of generality that $g_n\ne e,\,n=1,2,\dots$
Since $G$ is NSS, we can find and fix an open symmetric neighborhood $U$ of $e$ containing no nontrivial subgroups of $G$. Fix $n\in \mathbb N$. Since
$g_n\ne e$, the cyclic group $\langle g_n\rangle$ is non-trivial, so, $\langle g_n\rangle\not\subset U$; hence there exists $m_n\in \mathbb Z$ such that $g_n^{m_n}\not\in U$. Consequenty, we have constructed a sequence $(m_n)_{n\in \mathbb N}$ such that
$$
g_n^{m_n}\not\in U,\,\,n=1,2,\dots
$$
This means that the sequence $(g_n)_{n\in \mathbb N}$ is not hyper-convergent to $e$.
\item[(b)] follows from (a) by Lemma \ref{co3}$(a)$.
\item[(c)] Let $(D_i)_{i\in I}$ be a sequence of  non-trivial finite discrete abelian groups and   $H$ be $\prod_{i\in \mathbb N}^{({\rm r})}D_i$ with the topology induced by $\prod_{i\in \mathbb N}D_i$. Then $H$ is an  abelian countable metrizable precompact group, $H\in {\rm HTAP}$ (by Proposition \ref{agv14}(3)), but  $H\not\in {\rm NSS}$ (because $\mathcal N(H)$  has a countable basis consisting of open subgroups).
\end{itemize}
 \end{proof}
\begin{lemma}\label{co14.1}
If  $G\not\in {\rm NSS}$ and $d:G\times G\to \mathbb R_+$ is a continuous mapping with $d(e,e)=0$, then there exists a sequence $(g_n)_{n\in \mathbb N}$ of elements of $G\setminus\{e\}$ such that 
 \
\begin{equation}\label{bepe}
\sum_{n\in \mathbb N}d(e, g_n^{m_n})<\infty,\quad \forall (m_n)_{n\in \mathbb N}\in \mathbb Z^\mathbb N\,.
\end{equation}

\end{lemma}
\begin{proof}
 Write:
$$
U_n=\{g\in G: d(e,g)<\frac{1}{2^n}\},\,\,n=1,2,\dots
$$
Since $d$ is continuous on $G\times G$, we have
$$
U_n\in \mathcal N(G),\,\,n=1,2,\dots
$$
Since $G\not\in {\rm NSS}$, there is a sequence $(H_n)_{n\in \mathbb N}$ of {\it non-trivial} subgroups of $G$ such that
$$
H_n\subset U_n,\,\,\,n=1,2,\dots
$$
Hence there exists a sequence $(g_n)_{n\in \mathbb N}$ of elements of $G$ such that
\begin{equation}\label{bepe1}
g_n\ne e,\quad g_n\in H_n\subset U_n,\,\,\,n=1,2,\dots
\end{equation}
Fix a sequence $(m_n)_{n\in \mathbb N}$ of integers. From (\ref {bepe1}), as $H_n,\,n=1,2,\dots$ are subgroups of $G$, we get:
\begin{equation}\label{bepe2}
 g_n^{m_n}\in H_n\subset U_n,\,\,\,n=1,2,\dots
\end{equation}
Hence,
\begin{equation}\label{bepe3}
d(e, g_n^{m_n})<\frac{1}{2^n},\,\,n=1,2,\dots
\end{equation}
From (\ref {bepe3}) it follows that (\ref {bepe}) is true.
\end{proof}
From Theorem \ref{shakh}(a) and Lemma \ref{co14.1} we deduce:
\begin{theorem}\label{shakh2}
For a metrizable group $G$ TFAE:
\begin{itemize}
\item[(i)] $G\in {\rm NSS}$.
\item[(ii)] $ G\in {\rm STAP}$.
\end{itemize}
\end{theorem}
\begin{proof}
$(i)\Longrightarrow (ii)$ by Theorem \ref{shakh}$(a)$.\\
$(ii)\Longrightarrow (i)$. Suppose that $G\not \in {\rm NSS}$. Take some metric $d$ which metrizes  $G$. Then by Lemma \ref{co14.1} there exists a not eventually neutral sequence   $(g_n)_{n\in \mathbb N}$ of elements of $G$ for which (\ref {bepe}) is satisfied for  $d$.
Since $d$ metrizes  $G$, from (\ref {bepe}) it follows that $(g_n)_{n\in \mathbb N}$ hyper-converges to $e$. Consequently, $G\not\in {\rm STAP}$.
 \end{proof}
We will derive the following result from  Theorem \ref{shakh2} and Lemma \ref{co14.1}.

\begin{theorem}\label{shakh3}
For a {\em Weil-complete} metrizable group $G$ TFAE:
\begin{itemize}
\item[(i)] $G\in {\rm NSS}$.
\item[(ii)] $G\in {\rm STAP}$.
\item[(iii)] $G\in {\rm HTAP}$.
\item[(iv)] $ G\in {\rm TAP}$.
\end{itemize}
\end{theorem}
\begin{proof}
The equivalence (i)$\Longleftrightarrow$(ii) follows from Theorem \ref{shakh2}.\\
The implications (ii)$\Longrightarrow$(iii)$\Longrightarrow$ (iv) is evident.\\
 (iv)$\Longrightarrow$(i).\\
Suppose that $G\not\in {\rm NSS}$. Take some {\it left invariant} metric $d$ which metrizes  $G$
 (such a metric exists by Birkhoff-Kakutani's theorem). Since $G\not\in {\rm NSS}$, by Lemma \ref{co14.1} there exists a {\it not eventually neutral sequence} $(g_n)_{n\in \mathbb N}$ of elements of $G$ satisfying (\ref{bepe}). Clearly,  the range $A:=\{g\in G: \exists n\in \mathbb N,\,g=g_n\}$ of $(g_n)_{n\in \mathbb N}$ is infinite.
 \par
 Now we will show that the infinite  set $A$ is absolutely productive in $G$ and this will show that $G\not\in  {\rm TAP}$.
 \par 
  Consider an injection $a:\mathbb N\to A$,  a mapping $z:\mathbb N\to \mathbb Z$ and let us see that  the sequence
 $$
 \left( \Pi_{k=1}^n a(k)^{z(k)} \right)_{n\in \mathbb N}
 $$
converges to  some $g\in G$.
\par
For every $n\in \mathbb N$ we can find and fix some $k_n\in \mathbb N$ such that $a(n)=g_{k_n}$. Since $a$ is injective, thus obtained sequence $(k_n)_{n\in \mathbb N}$ consists of pairwise distinct natural numbers. Let $(m_n)_{n\in \mathbb N}$ be a sequence of integers defined as follows:
$m_{k_n}=z(n),\,n=1,2,\dots$ and $m_j=0,\forall j\in \mathbb N\setminus\{k_1,k_2,\dots\}$.
\par
Observe that
$$
\sum_{j\in \mathbb N}d(e, g_j^{m_j})=\sum_{n\in \mathbb N}d(e, a(n)^{z(n)})\,.
$$
From this and (\ref{bepe}) we get:
\begin{equation}\label{14gan1}
\sum_{n\in \mathbb N}d(e, a(n)^{z(n)})<\infty.
\end{equation}
  
 Write:
$$
b_n=\prod_{k=1}^n a(k)^{z(k)},\,n=1,2,\dots
$$
Since $d$ is left-invariant, we have:
$$
d(b_{n-1},b_n)=d(b_{n-1},b_{n-1}  a(n)^{z(n)})=d(e, a(n)^{z(n)}),\,\,n=2,3,\dots
$$
This and (\ref{14gan1}) imply:
\begin{equation}\label{16.1}
\sum_{n=2}^{\infty}d(b_{n-1},b_n)<\infty
\end{equation}
From (\ref{16.1}) we get that $(b_n)_{n\in \mathbb N}$ is a $d$-Cauchy sequence. From this, since $d$ is left-invariant and metrizes $G$, we get:
\begin{equation}\label{16.2}
\forall V\in\mathcal N(G)\,\exists n_V\in \mathbb N,\,\,[n',n''\in \mathbb N,\,\min(n',n'')\ge n_V\,\Longrightarrow\,b_{n'}^{-1}b_{n''}\in V]\,.
\end{equation}
From (\ref{16.2}) we conclude that $(b_n)_{n\in \mathbb N}$ is a Cauchy sequence in the left uniformity of $G$. From this, since $G$ is Weil complete, we get that 
$(b_n)_{n\in \mathbb N}$ converges to some $g\in G$. Since $a:\mathbb N\to A$ and  $z:\mathbb N\to \mathbb Z$ are arbitrary, we have proved that 
the infinite set $A$ is absolutely productive in $G$.\end{proof}
\begin{remark}\label{ars-sh}{\em (1) The validitity of the implication $(iv)\Longrightarrow (i)$ of Theorem \ref{shakh3} is stated in \cite {tapstap},  however it is proved therein  only the weaker implication $(iii)\Longrightarrow (i)$. This gap was observed and was pointed out to us by professor D. Dikranjan. A different proof of the implication\\
$G$ is Weil complete metrizable \,$+$\,$G\in {\rm TAP}$\,$\Longrightarrow$\,$G\in {\rm NSS}$\\
is contained in the proof of \cite[Corollary 4.6]{DSS}.
\par 
 (2)
The implication $G\in {\rm HTAP}\Longrightarrow G\in {\rm NSS}$ of Theorem \ref{shakh3} fails drastically for a non-complete metrizable abelian $G$ (see Theorem \ref{shakh}(c)), as well as for a {\it Raikov complete} metrizable $G$ 
\cite[Remark 5.9]{DSS}.
\par
(3) In \cite[Theorem 9.3]{DSS} it is proved
that a $\sigma$-compact complete abelian TAP group need not be NSS. Therefore, the implication $(iv)\Longrightarrow (i)$ of Theorem \ref{shakh3} may fail without the metrizability assumption.
}
\end{remark}
\par
From Theorem \ref{shakh3}, since locally compact groups are Weil-complete, we get
\begin{Cor}\label{shakh4}
For a {\em locally compact} metrizable group $G$ we have that $G\in {\rm NSS}\Longleftrightarrow G\in {\rm TAP}$.
\end{Cor}
\begin{remark}\label{no-mert}{\em
Corollary \ref{shakh4} remains true without metrizability:   it has been proved recently in  \cite[Theorem 10.8]{DSS} that a locally compact $\rm{TAP}$ group is $\rm{NSS}$.
}

\end{remark}
 

\section{$C_p(X)$ and ${\rm STAP}$}\label{cepe}

In this subsection $X$ will be a Tychonoff space and $C_p(X,G)$ will stand for the group of all continuous mappings $f:X\to G$ endowed with the topology of pointwise convergence.
\begin{theorem}\label{pseudo}{\em \cite[Theorem 5.3]{SS}}
A space $X$ is pseudocompact if and only if $C_p(X,\mathbb R)$ has the ${\rm TAP}$ property.
\end{theorem}
To prove a ${\rm STAP}$-version of this result we need a lemma.
\begin{lemma}\label{Jale}
Let $Y$ be an infinite topological space.
\par
(a) If $Y$ is Hausdorff, then there exists $y\in Y$ and a sequence $(y_n)_{n\in \mathbb N}$ of distinct points of $Y\setminus \{y\}$ which does not converge to $y$.
\par
$(b)$ If $Y$ is Hausdorff regular, then there exists a sequence $(V_n)_{n\in \mathbb N}$ of {\em non-empty pairwise disjoint open} subsets of $Y$. 
\end{lemma} 
\begin{proof}
$(a)$ 
Pick $y\in Y.$ Choose a sequence $(y_n)_{n\in {\mathbb N}}$  of pairwise distinct  elements of $Y\setminus\{y\}$. If $y_n\not \to y$ we are done.
Otherwise let $y$ and $y_1$ change places. Then, since $Y$ is Hausdorff, the sequence $y,y_2,y_3,\dots$ cannot converge to $y_1\ne y$.
\par
$(b)$ By (a) we can choose and fix a point $b_1\in Y$ and a sequence $(y_{1,n})_{n\in \mathbb N}$ of distinct points of $Y\setminus \{b_1\}$ which does not converge to $b_1$. Let $V_1$ be an open neighbourhood of $b_1$ for which the set $N_1:=\{n\in \mathbb N: y_{1,n}\not\in \overline{V_1}\}$ is infinite (such a $V_1$ exists because $Y$ is regular).\\
Let $Y_1:=\{y_{1,n}: n\in N_1\}$. Apply now (a) to the Hausdorf space $Y_1$ and find $b_2\in Y_1$ and a sequence $(y_{2,n})_{n\in \mathbb N}$ of distinct points of $Y_1\setminus \{b_2\}$ which does not converge to $b_2$. Since $b_2\in Y\setminus \overline{V_1}$, the set $Y\setminus \overline{V_1}$ is open, $Y$ is regular  and $(y_{2,n})_{n\in \mathbb N}$ does not converge to $b_2$, we can find an open neighbourhood $V_2$ of $b_2$ for which  $\overline{V_2}\subset Y\setminus \overline{V_1}$ and  the set $N_2:=\{n\in \mathbb N: y_{2,n}\not\in \overline{V_2}\}$ is infinite.
\par
Clearly, $V_1$ and $V_2$ are open sets with disjoint closures.
\par
Let $Y_2:=\{y_{2,n}: n\in N_2\}$. Apply again  (a) to the Hausdorf space $Y_2$ and find $b_3\in Y_2$ and a sequence $(y_{3,n})_{n\in \mathbb N}$ of distinct points of $Y_2\setminus \{b_3\}$ which does not converge to $b_3$. Since $b_3\in Y\setminus (\overline{V_1}\cup \overline{V_2})$, the set $Y\setminus (\overline{V_1}\cup \overline{V_2})$ is open, $Y$ is regular  and $(y_{3,n})_{n\in \mathbb N}$ does not converge to $b_3$, we can find an open neighbourhood $V_3$ of $b_3$ with   $\overline{V_3}\subset Y\setminus (\overline{V_1}\cup \overline{V_2})$ and for which the set $N_3:=\{n\in \mathbb N: y_{3,n}\not\in \overline{V_3}\}$ is infinite.
\par
Clearly, $V_1$ $V_2$ and $V_3$ are open sets with pairwise disjoint closures.
\par
Inductively, we  thus  obtain a sequence $(V_n)_{n\in \mathbb N}$ of { non-empty pairwise disjoint open} subsets of $Y$ (for which the sequence $(\overline{V_n})_{n\in \mathbb N}$ is pairwise disjoint as well).
\end{proof}
\begin{remark}{\em The statement of Lemma \ref{Jale}$(b)$ is contained in the proof of \cite [p. 233, Lemma 11.7.1]{Jar}. For the sake of self-containedness we have separated this statement and have reproduced its proof as well. Note that from Lemma \ref{Jale}$(b)$ it can be concluded that the cardinality of  the topology of an infinite regular Hausdorff space is greater than the cardinality of continuum.}
\end{remark}

\begin{theorem}\label{pseudo2m}
Let $X$ be a Tychonoff space. Then $C_p(X,\mathbb R)$ has the ${\rm STAP}$ property if and only if $X$ is finite.
\end{theorem}
\begin{proof}
Suppose that $X$ is infinite and let us show that then $C_p(X,\mathbb R)\not\in {\rm STAP}$.\\
Since $X$ is infinite, by Lemma \ref{Jale}$(b)$ there exists a sequence $(V_n)_{n\in \mathbb N}$ of {\it non-empty pairwise disjoint open} subsets of $X$. Fix a sequence $(x_n)_{n\in \mathbb N}$ of (pairwise distinct) elements of $X$ such that
\begin{equation}\label{a25eqS}
x_n\in V_n,\,\,\, n=1,2,\dots
\end{equation}
Using (\ref{a25eqS}) and the complete regularity of $X$ we can find a sequence continuous functions  $f_n:X\to[0,1],\,\,n=1,2,\dots$ such that
\begin{equation}\label{a25eq3}
f_n(x_n)=1,\,\,f_n(x)=0,\,\forall x\in X\setminus V_n,\,\,\,n=1,2,\dots
\end{equation}
From (\ref{a25eq3})  it follows in particular  that $f_n\ne 0,\,\,n=1,2,\dots$, where $0$ stands for the identically zero function: the neutral element of $C_p(X,\mathbb R)$.
Consequently, $(f_n)_{n\in \mathbb N}$ {\it is not an eventually neutral sequence in} $C_p(X,\mathbb R)$. Now we will show that $(f_n)_{n\in \mathbb N}$ is a hyper-null sequence in $C_p(X,\mathbb R)$ and thus $C_p(X,\mathbb R)\not\in {\rm STAP}$.
\par 
Since the sets $V_n,\,n=1,2,\dots$ are pairwise disjoint, we have:
\begin{equation}\label{a25eq2}
{\rm Card}\{n\in \mathbb N:x\in V_n\}\le 1\,\,\forall x\in X\,.
\end{equation}
\par
Fix $x\in X$ and let us derive now from (\ref{a25eq2}) that $(f_n(x))_{n\in \mathbb N}$ {\it is  an eventually neutral sequence in} $\mathbb R$.\\
If $x\in X\setminus \cup_{n\in \mathbb N} V_n$, then from (\ref{a25eq3}) we have: $f_n(x)=0,\,n=1,2,\dots$\\
If $x\in  \cup_{n \in \mathbb N}V_n$, then $x\in V_{n_x}$ for some $n_x\in \mathbb N$.
This, together with (\ref{a25eq3}), gives:
$$
n\in\mathbb N,\,\,n>n_x\,\,\Longrightarrow\,\,f_n(x)=0\,.
$$
Hence $(f_n(x))_{n\in \mathbb N}$ {\it is  an eventually zero sequence in} $\mathbb R$. This clearly implies that for {\em every} sequence $(m_n)_{n\in \mathbb N}$ extracted from $\mathbb Z$ the sequence $(m_nf_n(x))_{n\in \mathbb N}$ is again an  eventually zero sequence in $\mathbb R$. In particular, we have that 
 $\lim_nm_nf_n(x)=0$ for {\em every} sequence $(m_n)_{n\in \mathbb N}$ extracted from $\mathbb Z$ and this by the definition of the topology of $C_p(X,\mathbb R)$ means that $(f_n)_{n\in \mathbb N}$ is a hypernull sequence in $C_p(X,\mathbb R)$.
 \par
 So, we have proved that if $C_p(X,\mathbb R)$ has the ${\rm STAP}$ property, then $X$ is finite. Conversely, if $X$ is finite, then clearly $C_p(X,\mathbb R)=\mathbb R^X \in {\rm STAP}$. \end{proof}

\section{The case of topological vector groups}\label{tvp}
In this section we are going to demonstrate that the methods of \cite{BPR} can be used to characterize the (complete) metrizable real topological vector spaces and topological vector groups which are $\rm{STAP}$ ($\rm{TAP}$). 
\par
A nonempty subset $A$ of a real vector space $E$ is called {\it balanced} if
$$
t\in \mathbb R,\,|t|\le 1\,\Longrightarrow\,tA:=\{ta:\,a\in A\}\subset A\,.
$$
If $A$ is balanced, then $0\in A$ and $A$ is symmetric. If  $0\in A$ and $A$ is symmetric and $A$ is convex, then $A$ is balanced (the converse is not true when ${\rm{Dim}}(E)>1$).
\begin{lemma}\label{1seq.1}
Let $A$ be a balanced subset of a real vector space $E$ and $a\in A$ be such that
$$
\langle a\rangle=\{ma:\,m\in \mathbb Z\} \subset A\,.
$$
Then 
$$
\mathbb Ra:=\{ta:\,t\in \mathbb R\}\subset A\,.
$$
\end{lemma}
\begin{proof}
Fix $t\in \mathbb R$. Find a natural number $n$ such that $|\frac{t}{n}|\le 1$. Take $a\in A$. Then $na\in \langle a\rangle\subset A$. Hence, as $na\in A$ and $A$ is balanced, we get: $ta=\frac{t}{n}(na)\in A$. \end{proof}
\begin{lemma}\label{1seq.2}
Let $E$ be a topological vector space over $\mathbb R$. Then:
\par
$(a)$ For every $t\in \mathbb R\setminus \{0\}$ the map $x\mapsto tx$ is a linear homeomorphism of $E$ onto $E$.
\par
$(b)$ The balanced members of $\mathcal N(E)$ form a basis of $\mathcal N(E)$.
\end{lemma}
\begin{proof}
$(a)$ is easy to see.\\
$(b)$ Fix  $U\in \mathcal N(E)$. Since the mapping $(t,x)\mapsto tx$ is continuous at $(0,0)\in \mathbb R\times E$, we can find $\varepsilon>0$ and $V\in \mathcal N(E)$ such that $t\in [-\varepsilon,\varepsilon],\,x\in V\,\Longrightarrow\,tx\in U$. Equivalently, for $\varepsilon>0$ and $V\in \mathcal N(E)$ we have $$[-\varepsilon,\varepsilon]V:=\{tx:\,t\in [-\varepsilon,\varepsilon],\,x\in V\,\}\subset U\,.$$ It is easy to see that the set 
$[-\varepsilon,\varepsilon]V$ is balanced. From $(a)$ it follows that  $\varepsilon V\in \mathcal N(E)$. Since $\varepsilon V\subset [-\varepsilon,\varepsilon]V$, we get $[-\varepsilon,\varepsilon]V\in \mathcal N(E)$. Consequently, for $U\in \mathcal N(E)$ we have found a balanced 
$[-\varepsilon,\varepsilon]V\in \mathcal N(E)$ such that $[-\varepsilon,\varepsilon]V\subset U$.\end{proof}

A topological abelian group $G$ which is also a vector space over $\mathbb R$ is called {\it a topological vector group over $\mathbb R$} if for every $t\in \mathbb R$ the map $x\mapsto tx$ is continuous.
\par
A topological abelian group $G$ which is also a vector space over $\mathbb R$ is called {\it locally balanced} if the balanced members of $\mathcal N(E)$ form a basis of $\mathcal N(E)$.
\par
A topological abelian group $G$ which is also a vector space over $\mathbb R$ is called {\it locally convex} if the convex symmetric members of $\mathcal N(E)$ form a basis of $\mathcal N(E)$.
\par
\begin{lemma}\label{1seq.3}
Let $G$  be a topological abelian group which is also a vector space over $\mathbb R$. Then
\par
$(a)$ If $G$ is locally balanced, then $G$ is a topological vector group over $G$.
\par
$(b)$ If $G$ is locally convex, then $G$ is a topological vector group over $G$.
\end{lemma}
\begin{proof}
$(a)$ Fix $t\in \mathbb R$ and $U\in \mathcal N(E)$. If $t\in [-1,1]$, find a balanced $V\in \mathcal N(E)$ with $V\subset U$. Then $tV\subset V\subset U$. Hence the map $x\mapsto tx$ is continuous at $0\in E$ and hence, is continuous. If $|t|>1$, we can write: $t=s+m$ for some $s\in [-1,1]$ and $m\in \mathbb Z$. Then, as we have seen,  the map $x\mapsto sx$ is continuous; since $G$ is a  topological group, the map  $x\mapsto mx$  is continuous as well. Consequently the map  $x\mapsto tx$, as the pointwise sum of two continuous mappings, is continuous as well.
\par
$(b)$ follows from $(a)$. \end{proof}
We will not need much about topological vector groups. For the reader it would be sufficient to believe that the class of locally balanced topological vector groups is wider that the class of topological vector spaces.
\par
We begin with the following easy observation.
\begin{Pro}\label{tvsnss}
For a locally balanced  Hausdorff topological vector group  $G\ne \{0\}$ over $\mathbb R$ TFAE:
\par
$(i)$ $G\not\in \rm{NSS}$.
\par
$(ii)$ Every  $U\in \mathcal N(G)$ contains a $1$-dimensional vector subspace of $G$.
\end{Pro}
\begin{proof}
$(i)\Longrightarrow (ii)$. Fix $U\in \mathcal N(G)$, find a balanced $V\in\mathcal N(E)$ with $V\subset U$. By $(i)$ there is a nontrivial subgroup $H$ of $G$, such that $H\subset V$. Take some $x\in H\setminus \{0\}$. Then $\langle x\rangle\subset H\subset V$. From this and Lemma \ref{1seq.1} we get:
$\mathbb R x\subset V\subset U$.
\par
$(ii)\Longrightarrow (i)$ is true because a $1$-dimensional vector subspace of $G$ is also a nontrivial subgroup of $G$.
\end{proof}
It is clear that if $X$ is an infinite set,  then $\mathbb R^X\not\in \rm{NSS}$. This follows also from the following statement, which is not evident at once.
\begin{Pro}\label{tvnss2}
Let $X$ be a non-empty set and $E\ne \{0\}$ be a vector subspace of $\mathbb R^X$ endowed with the induced topology.
\par
If $E\in \rm{NSS}$, then $\rm{Dim}(E)<\infty$.
\end{Pro}
\begin{proof}
For a non-empty subset $\Delta\subset X$, $\varepsilon >0$  and write:
$$
V_{\Delta,\varepsilon}:=\{f\in E: f(\Delta)\subset [-\varepsilon,\varepsilon]\}.
$$
The sets $V_{\Delta,\varepsilon},\,\,\,\varepsilon>0$ when $\Delta$ runs over all finite subsets of $X$ form a basis for $\mathcal N(E)$.
\par
As $E\in \rm{NSS}$, by Proposition  \ref{tvsnss} there exists a finite non-empty $\Delta\subset X$ and $\varepsilon >0$, such that $V_{\Delta,\varepsilon }$ does not contain any $1$-dimensional vector subspace of $E$. Consider now the  mapping $u:E\to \mathbb R^{\Delta}$ defined by: $u(f)=f|_{\Delta},\,\,f\in E$. Clearly $u$ is linear. Let us see that $u$ is injective too. Take $f\in E$ with $u(f)=0\in \mathbb R^{\Delta}$. This means that $f(\Delta)=\{0\}$. Then $tf(\Delta)=\{0\},\,\,\forall t\in \mathbb R$. Hence, $tf \in V_{\Delta,\varepsilon},\,\,\forall t\in \mathbb R $. From this, since $V_{\Delta,\varepsilon}$ does not contain any $1$-dimensional vector subspace of $E$, we get that $f=0$. Therefore $u$ is injective and so, it is a vector space isomorphism between $E$ and $u(E)\subset \mathbb R^{\Delta}$. Consequently,  $\rm{Dim}(E)=\rm{Dim}(u(E))\le \rm{Dim}(\mathbb R^{\Delta})=\rm{Card}(\Delta)<\infty$.
 \end{proof}
 The formulation of the following statement for the case of complete metrizable topological vector spaces is contained in the proof of  \cite[Theorem 9]{BPR}.
 \begin{lemma}\label{1seq.4}
 Let $G\ne \{0\}$  be a locally balanced  Hausdorff topological vector group   over $\mathbb R$   and $d$ be a continuous metric on $G$. 
 \par
 If $G\not\in \rm{NSS}$, then there exists a sequence $(g_n)_{n\in \mathbb N}$ of  elements of $G$ such that
 
 \begin{equation}\label{1segan.1}
 0<\sup_{t\in \mathbb R}d(0,tg_1)<1,\,\, \sup_{t\in \mathbb R}d(0,tg_{n+1})<\frac{1}{4}\sup_{t\in \mathbb R}d(0,tg_n),\, n=1,2,\dots
 \end{equation}
  
 \end{lemma}
 \begin{proof}
 For a real number $r,\,0<r<1$ write:
$$
U_r=\{g\in G: d(0,g)<r\}\,.
$$
Since $d$ is continuous on $G\times G$, we have
$$
U_r\in \mathcal N(G)\,.
$$
Since $G\not\in \rm{NSS}$, by Proposition \ref{tvsnss} we can find and fix $g_1\in U_1\setminus \{0\}$ such that $\mathbb Rg_1\subset U_r$. Set
$$
\delta_1:=\sup_{t\in \mathbb R}d(0,tg_1)\,.
$$
Clearly, $0<\delta_1\le r<1$.
\par
In the same way we can find and fix $g_2\in U_{\frac{\delta_1}{5}}\setminus \{0\}$ such that $\mathbb Rg_2\subset U_{\frac{\delta_1}{5}}$. Set
$$
\delta_2:=\sup_{t\in \mathbb R}d(0,tg_2)\,.
$$
Clearly, $0<\delta_2\le \frac{\delta_1}{5}<\frac{\delta_1}{4}$.
\par
Suppose now that for a natural number $n\ge 2$ the elements $g_2,\dots,g_n$  and the numbers $\delta_2,\dots,\delta_n$ are already constructed such that
$$
\delta_{k}<\frac{1}{4}\delta_{k-1},\,\,\,\delta_k=\sup_{t\in \mathbb R}d(0,tg_k),\, k=2,\dots,n\,.
$$
Since $G\not\in \rm{NSS}$, by Proposition \ref{tvsnss} again we can find and fix $g_{n+1}\in U_{\frac{\delta_n}{5}}\setminus \{0\}$  such that $\mathbb Rg_{n+1}\subset U_{\frac{\delta_n}{5}}$ and set
$$
\delta_{n+1}:=\sup_{t\in \mathbb R}d(0,tg_{n+1})\,.
$$
Then  $0<\delta_{n+1}\le \frac{\delta_n}{5}<\frac{\delta_n}{4}$.
\par
In this way we can construct the sequence satisfying (\ref{1segan.1}).
 \end{proof}
 For a sequence $\mathfrak g:=(g_n)_{n\in \mathbb N}$ of elements of a topological vector group $G$ write:
 $$
 E_{\mathfrak g}:=\{t\in \mathbb R^{\mathbb N}:\,\left(\sum_{i=1}^nt(i)g_i\right )_{n\in \mathbb N}\,\text{converges}\,\,\,\text{in}\,\,\, G\,\}
 $$
 and define a mapping $u_{\mathfrak g}:E_{\mathfrak g}\to G$ by the equality:
 $$
 u_{\mathfrak g}t=\sum_{i=1}^{\infty}t(i)g_i:=\lim_n\sum_{i=1}^nt(i)g_i,\,\, t\in E_{\mathfrak g}\,.
 $$
 This notation will be used in the following  statement which is a kind of converse to Lemma \ref{1seq.4}.
 \begin{Pro}\label{1seqmt1}
 Let $G$  be a locally balanced  metrizable  topological vector group   over $\mathbb R$  and $d$ be an invariant metric which metrizes $G$. 
 \par
 If there exists a sequence  $\mathfrak g:=(g_n)_{n\in \mathbb N}$ of elements of $G$ satisfying (\ref{1segan.1}), then:
  \par
 $(a)$ $E_{\mathfrak g}$ is a vector subspace of $\mathbb R^{\mathbb N}$, $\mathbb R^{(\mathbb N)}\subset E_{\mathfrak g}$ and $u:E_{\mathfrak g}\to G$ is a linear {\em injective} mapping.
 \par
 $(b)$ If $E_{\mathfrak g}$ is endowed by the induced from $\mathbb R^{\mathbb N}$ topology, then $u=u_{\mathfrak g}:E_{\mathfrak g}\to G$ is a {\em continuous} linear  injective mapping for which the inverse mapping $u^{-1}:u(E_{\mathfrak g})\to E_{\mathfrak g}$ {\em is continuous as well}.
 \par
 $(c)$ $G$ contains a vector subspace which is linearly homeomorphic to $\mathbb R^{(\mathbb N)}$ endowed with the topology induced from $\mathbb R^{\mathbb N}$.
 \par
 $(d)$ $G\not\in \rm{NSS}$.
 \par
 $(e)$ If $G$ is complete, then $E_{\mathfrak g}=\mathbb R^{\mathbb N}$ and hence, $G$ contains a vector subspace which is linearly homeomorphic to $\mathbb R^{\mathbb N}$
   \end{Pro}
  \begin{proof}
  Put:
  \begin{equation}\label{1segan.2}
 \delta_n:= \sup_{t\in \mathbb R}d(0,tg_n),\, n=1,2,\dots
 \end{equation}
  
  Since $\delta_{n+1}<\frac{1}{4}\delta_n,\, n=1,2,\dots$, for a fixed $n$ we have: $\delta_{n+k}<\frac{1}{4^k}\delta_n,\, k=1,2,\dots$,      
 consequently, 
  \begin{equation}\label{1segan.3}
   \sum_{i=n+1}^{\infty}\delta_i<\sum_{i=1}^{\infty}\frac{1}{4^{i}}\delta_n=\frac{1}{3}\delta_n, \,\, n=1,2,\dots\,
  \end{equation}
  $(a)$ All statements, except injectivity, are easy to verify and they are true for each sequence  $\mathfrak g=(g_n)_{n\in \mathbb N}$. The injectivity of $u$ will follow from relation (\ref{2seq-sag}), which will be proved below.
  
  \par
  $(b)$  First let us see that $u$ is continuous. Take a sequence $t_n\in E_{\mathfrak g},\,n=1,2,\dots$ with $\lim_nt_n(i)=0,\forall i\in \mathbb N$, fix $\varepsilon>0$. Since $\sum_{i=2}^\infty \delta_i\le \frac{1}{3}\delta_1<\infty$, for some $k\in \mathbb N$ we have that $\sum_{i=k}^\infty \delta_i<\varepsilon$. So, for a fixed $n\in \mathbb N$ we have:
  $$
  d(ut_n,0)=d(\sum_{i=1}^{k}t_n(i)g_i+\sum_{i=k+1}^{\infty}t_n(i)g_i, 0)\le d(\sum_{i=1}^{k}t_n(i)g_i,0)+d(\sum_{i=k+1}^{\infty}t_n(i)g_i, 0)\le
  $$
    $$
  d(\sum_{i=1}^{k}t_n(i)g_i,0)+\sum_{i=k+1}^{\infty}d(t_n(i)g_i, 0)\le d(\sum_{i=1}^{k}t_n(i)g_i,0)+\sum_{i=k+1}^{\infty}\delta_i<d(\sum_{i=1}^{k}t_n(i)g_i,0)+\varepsilon\,.
  $$
  Hence, as $\lim_n\sum_{i=1}^{k}t_n(i)g_i=0$, we obtain:
  $$
  \limsup_nd(ut_n,0)\le \limsup_n\left( d(\sum_{i=1}^{k}t_n(i)g_i,0)+\varepsilon \right)=\lim_nd(\sum_{i=1}^{k}t_n(i)g_i,0)+\varepsilon=\varepsilon\,.
    $$
  From this, since $\varepsilon>0$ is arbitrary, we get: $\lim_nd(ut_n,0)=0$. Hence, $u$ is continuous.
    \par
    Now we want to prove the following implication, from which will follow the injectivity of $u$ together with the  continuity of $u^{-1}:u(E_{\mathfrak g})\to E_{\mathfrak g}$.
    \begin{equation}\label{2seq-sag} 
    t_n\in E_{\mathfrak g},\,n=1,2,\dots;\, \lim_n ut_n=0\,\,\Longrightarrow\, \lim_nt_n(i)=0,\forall i\in \mathbb N\,.
    \end{equation}
     
  To prove (\ref{2seq-sag}), we can assume without loss of generality that the metric $d$ has the following additional property:
  \begin{equation}\label{2segan.1}
  g\in G,\,\alpha_1,\alpha_2\in \mathbb R,\,\, |\alpha_1|\le |\alpha_2|\,\Longrightarrow\,d(\alpha_1\cdot g,0)\le d(\alpha_2\cdot g,0)\,.
  \end{equation}
    Fix a sequence $t_n\in E_{\mathfrak g},\,n=1,2,\dots$ such that
    \begin{equation}\label{2segan.2}
  \lim_n ut_n=\lim_n \sum_{i=1}^{\infty}t_n(i)g_i=0\,.
  \end{equation}
  and derive from (\ref{2segan.2}) that 
  \begin{equation}\label{2segan.3}
  \lim_nt_n(i)=0,\forall i\in \mathbb N\,.
  \end{equation}
   Since $u$ is homogeneus, (\ref{2segan.2}) implies:
  \begin{equation}\label{2segan.4}
  \lim_n u(\alpha\cdot t_n)=\lim_n\sum_{i=1}^{\infty}\alpha \cdot t_n(i)g_i=0\,\,\,\forall\alpha\in \mathbb R\,.
  \end{equation}
  Let us show first that $\lim_nt_n(1)=0$.   
  Suppose that the sequence $(t_n(1))_{n\in \mathbb N}$ does not tend to $0$. Then for some $\eta>0$ and some strictly increasing sequence $(k_n)$ of natural numbers we shall have:
  \begin{equation}\label{2segan.5}
  |t_{k_n}(1)|>\eta>0,\,\,n=1,2,\dots
  \end{equation}
   As 
   $$ \alpha\cdot t_{k_n}(1)g_1= u(\alpha\cdot t_{k_n})-\sum_{i=2}^{\infty}\alpha\cdot t_{k_n}(i)g_i,\,\,n=1,2,\dots;\,\forall\alpha\in \mathbb R\,,$$
      we can write:
  \begin{equation}\label{2segan.6}
  d(\alpha\cdot t_{k_n}(1)g_1,0)\le d(u(\alpha\cdot t_{k_n}),0)+d(\sum_{i=2}^{\infty}\alpha\cdot t_{k_n}(i)g_i,0)<d(u(\alpha\cdot t_{k_n}),0)+\frac{1}{3}\delta_1,\,n=1,2,\dots;\,\forall\alpha\in \mathbb R\,
  \end{equation}
  Now from (\ref{2segan.6}) by using (\ref{2segan.1}) we get:
   \begin{equation}\label{2segan.7}
  d(\alpha\cdot \eta g_1,0)<d(u(\alpha\cdot t_{k_n}),0)+\frac{1}{3}\delta_1,\,n=1,2,\dots;\,\forall\alpha\in \mathbb R\,
  \end{equation}
  As (see (\ref{2segan.4}))
  $$\lim_n d(u(\alpha\cdot t_{k_n}),0)=0\,\,\,\forall\alpha\in \mathbb R\,$$
  from (\ref{2segan.7}) we get:
  \begin{equation}\label{2segan.8}
  d(\alpha\cdot \eta g_1,0)\le \frac{1}{3}\delta_1,\,n=1,2,\dots;\,\forall\alpha\in \mathbb R\,
  \end{equation}
  From (\ref{2segan.8}), as $\eta\ne 0$, we conclude:
  $$
  \delta_1=\sup_{\alpha\in \mathbb R}d(\alpha \cdot \eta g_1,0)\le \frac{1}{3}\delta_1\,.
    $$
  A contradiction. Hence, $\lim_nt_n(1)=0$.
  \par
    Suppose now that for a natural number $q$ we have already proved that $\lim_nt_n(i)=0,\,i=1,\dots,q$ and let us derive from this that $\lim_nt_n(q+1)=0$. Suppose again that this is not so, i.e., the sequence $(t_n(q+1))_{n\in \mathbb N}$ does not tend to $0$. Then for some $\eta>0$ and some strictly increasing sequence $(k_n)$ of natural numbers we shall have:
  \begin{equation}\label{2segan.9}
  |t_{k_n}(q+1)|>\eta>0,\,\,n=1,2,\dots
  \end{equation}
  For a fixed $n$ and $\alpha\in \mathbb R$ we have:
  $$
  \alpha\cdot t_{k_n}(q+1)g_{q+1}=u(\alpha\cdot t_{k_n})-\sum_{i=1}^{q}\alpha\cdot t_{k_n}(i)g_i-\sum_{i=q+2}^\infty \alpha\cdot t_{k_n}(i)g_i\,,
  $$
  from this, by the triangle inequality, we get:
  $$
  d(\alpha\cdot t_{k_n}(q+1)g_{q+1},0)\le d(u(\alpha\cdot t_{k_n}),0) +d(\sum_{i=1}^q\alpha\cdot t_{k_n}(i)g_i,0) +d(\sum_{i=q+2}^\infty \alpha\cdot t_{k_n}(i)g_i,0)\le
  $$
  $$
  d(u(\alpha\cdot t_{k_n}),0) +d(\sum_{i=1}^q\alpha\cdot t_{k_n}(i)g_i,0) +\sum_{i=q+2}^\infty \delta_i\,.
   $$
   Hence,
  \begin{equation}\label{2segan.10}
    d((\alpha\cdot t_{k_n}(q+1)g_{q+1},0)< d(u(\alpha\cdot t_{k_n},0)+ d(\sum_{i=1}^q\alpha\cdot t_{k_n}(i)g_i,0)  +\frac{1}{3}\delta_{q+1}\,,\,n=1,2,\dots;\,\forall\alpha\in \mathbb R\,
  \end{equation}
  Now from (\ref{2segan.10}) by using (\ref{2segan.1}) we get:
   \begin{equation}\label{2segan.11}
  d(\alpha\cdot \eta g_{q+1},0)<d(u(\alpha\cdot t_{k_n},0)+ d(\sum_{i=1}^q\alpha\cdot t_{k_n}(i)g_i,0)  +\frac{1}{3}\delta_{q+1}\,,\,n=1,2,\dots;\,\forall\alpha\in \mathbb R\,
  \end{equation}
  By our assumption we have that  $\lim_nt_n(i)=0,\,i=1,\dots,q$. This gives:
  $$
  \lim_n d(\sum_{i=1}^q\alpha\cdot t_{k_n}(i)g_i,0)=0,\forall\alpha\in \mathbb R\, .
  $$
  We have also    (see (\ref{2segan.4}):
  $$\lim_n d(u(\alpha\cdot t_{k_n}),0)=0\,\,\,\forall\alpha\in \mathbb R\,.$$
  The last two relations together with 
   (\ref{2segan.11}) imply:
  \begin{equation}\label{2segan.12}
  d(\alpha\cdot \eta g_{q+1},0)\le \frac{1}{3}\delta_{q+1},\,\forall\alpha\in \mathbb R\,
  \end{equation}
  From (\ref{2segan.8}), as $\eta\ne 0$, we conclude:
  $$
  \delta_{q+1}=\sup_{\alpha\in \mathbb R}d(\alpha \cdot \eta g_{q+1},0)\le \frac{1}{3}\delta_{q+1}\,.
    $$
  A contradiction. Hence, $\lim_nt_n(q+1)=0$. Consequently,  (\ref{2seq-sag}) is proved and it implies the needed injectivity of $u$ and  the continuity of $u^{-1}$. 
  \par
  $(c)$ follows from $(b)$ and $(d)$ follows from $(c)$.
  \par
  (e)  
   Fix $ t\in \mathbb R^{\mathbb N}$. Clearly,
  $$
  \sum_{n=2}^\infty d(t_ng_n,0)\le \sum_{n=2}^\infty \delta_n <\frac{1}{3}\delta_1<\infty\,.
  $$ 
  From this, since $E$ is complete, we get that the series $\sum_{n=}^\infty t_ng_n$ converges in $G$. Hence, $t\in E_{\mathfrak g}$. The rest follows now from $(b)$.
     \end{proof}
     
      Now we can prove the following characterization theorem, in which $\mathbb Z^{(\mathbb N)}$ and $\mathbb R^{(\mathbb N)}$ are supposed to be endowed with the topologies induced from $\mathbb R^{\mathbb N}$
\begin{theorem}\label{2seqTh1}
For a  metrizable locally balanced topological vector group $G$ over $\mathbb R$ TFAE:
\begin{itemize}
\item[(i)] $G\in \rm{NSS}$.
\item[(ii)]  $G\in \rm{STAP}$.
\item[(iii)]  $G$ does not contain a subgroup  topologically isomorphic to $\mathbb Z^{(\mathbb N)}$.
\item[(iv)]  $G$ does not contain  a vector subspace topologically isomorphic to $\mathbb R^{(\mathbb N)}$.
\end{itemize}
\end{theorem}
\begin{proof}
The equivalence $(i)\Longleftrightarrow (ii)$ is contained in Theorem \ref{shakh2}.\\
Implication $(ii)\Longrightarrow (iii)$ is true because $\mathbb Z^{(\mathbb N)}$ is not $\rm{STAP}$.\\
Implication $(iii)\Longrightarrow (iv)$ is evident.\\
$(iv)\Longrightarrow (i)$. Suppose that this implication is not true. Then  $G\not \in \rm{NSS}$. From  this, according to Lemma \ref{1seq.4} and Proposition \ref{1seqmt1}$(c)$, it follows that $G$ contains a subspace isomorphic to $\mathbb R^{(\mathbb N)}$.  
 A contradiction with $(iv)$.
\end{proof}

In complete case we have also,
\begin{theorem}\label{2seqTh2}
For a complete metrizable locally balanced topological vector group $G$ over $\mathbb R$ TFAE:
\begin{itemize}
\item[(i)] $G\in \rm{NSS}$.
\item[(ii)]  $G\in \rm{STAP}$.
 \item[(iii)]  $G\in \rm{TAP}$.
 \item[(iv)]  $G$ does not contain a subgroup   topologically isomorphic to $\mathbb Z^{(\mathbb N)}$.
\item[(v)]  $G$ does not contain a vector subspace  topologically isomorphic to $\mathbb R^{(\mathbb N)}$.
 \item[(vi)] $G$ does not contain a subgroup  topologically isomorphic to $\mathbb Z^\mathbb N$.
\item[(vii)]  $G$ does not contain a vector subspace  topologically isomorphic to $\mathbb R^\mathbb N$.
\end{itemize}
\end{theorem}
\begin{proof}
The equivalences $(i)\Longleftrightarrow (ii)\Longleftrightarrow (iii)$  are contained in  Theorem \ref{shakh3}.\\
The equivalences $(ii)\Longleftrightarrow (iv)\Longleftrightarrow (v)$ we have by Theorem \ref{2seqTh1}.\\
Implications $(i)\Longrightarrow (vi)$ is true because $\mathbb Z^\mathbb N\not\in \rm{NSS}$.
Implication $(vi)\Longrightarrow (vii)$ is evident.\\
$(vii)\Longrightarrow (i)$. Suppose that this implication is not true. Then  $G\not \in \rm{NSS}$. From  this, according to Lemma \ref{1seq.4} and Proposition \ref{1seqmt1}$(e)$, it follows that $G$ contains a vector subspace topologically isomorphic to $\mathbb R^{\mathbb N}$.  
 A contradiction with (vii).\\
 
\end{proof}
\begin{remark}{\em
The implication $(vi)\Longrightarrow (iii)$ in Theorem \ref{2seqTh2} may fail in general: in  \cite[Example 12.1]{DSS} it is shown that the group $\mathbb Z_p$ of $p$-adic integers is not $\rm{TAP}$, but it does not even contains  a direct product $A\times B$ of two nontrivial subgroups equipped with the product topology.
\par
The equivalences $(i)\Longleftrightarrow (v)$ and $(i)\Longleftrightarrow (iii)$ of Theorem \ref{2seqTh2} in case of complete metrizable topological vector spaces concide with \cite[Theorem 9]{BPR} and with \cite[p. 49, Corollary to Theorem 9]{BPR}. An interesting  proof of the implication $(vii)\Longrightarrow (i)$ of Theorem \ref{2seqTh2} in case of complete metrizable locally convex topological vector spaces 
is contained also in \cite[p. 129, Theorem 7.2.7]{Jar}.
}
\end{remark}
\section{Appendix: multipliability in abelian case}
In this section $G$ will stand for a Hausdorff topological abelian group with internal operation $\cdot$ and  $I$ will be an infinite (not necessarily countable) set, $\mathfrak F(I)$ will denote the collection of all {\it finite} subsets of $I$.
\par
Note that, since $G$ is abelian, for each $\alpha \in \mathfrak F(I)$ and a family $(g_i)_{i\in \alpha}$ of elements of $G$ the product $\prod_{i\in \alpha}g_i$ is well-defined (we agree that if $\alpha=\emptyset$, then  $\prod_{i\in \alpha}g_i:=e$).
\par
Let $(g_i)_{i\in I}$ be a family of elements of $G$ and $g\in G$. We say that the $(g_i)_{i\in I}$ {\it is multipliable to } $g$ if 
for each $U\in \mathcal N(G)$ there is 
$\alpha_U\in  \mathfrak F(I)$ such that 
\begin{equation}\label{9seq1}
\alpha\in  \mathfrak F(I),\,\,\alpha\supset \alpha_U\,\Longrightarrow\, \prod_{i\in \alpha}g_i\in gU\,.
\end{equation}
A family $(g_i)_{i\in I}$ of elements of $G$ will be called {\it multipliable} if there exists an element $g\in G$ such that $(g_i)_{i\in I}$ is multipliable to $g$.
\begin{remark}\label{a22rn1}{\em In case when the internal operation is $+$, instead of "a multipliable family", the term "a summable family" is used in \cite{Bou1}, where this concept is defined by using of the notion of the section filter associated with the directed partially ordered set  $\left(\mathfrak F(I),\subset\right)$.
 Multipliable sequences in (not necessarily abelian) normed algebras are treated in \cite{Bou2}; the case of groups is considered  in  \cite{Const1, Const2}.
}
\end{remark}
\par
The following lemma can be derived easily from the assumption that $G$ is Hausdorff.
\begin{lemma}\label{alo1}
Let $(g_i)_{i\in I}$ be a multipliable family of elements of $G$, $g,h\in G$. If $(g_i)_{i\in I}$ is multipliable to $g$ and  $(g_i)_{i\in I}$ is multipliable to $h$, then $g=h$.
\end{lemma}
\begin{lemma}\label{alo2}
Let $(g_i)_{i\in \mathbb N}$ be a sequence of elements of $G$, $g\in G$. If $(g_i)_{i\in \mathbb N}$ is multipliable to $g$ and  $\pi:\mathbb N \to \mathbb N$  is a bijection, then the sequence $(\Pi_{i=1}^ng_{\pi(i)})_{n\in \mathbb N}$ converges to $g$.
\end{lemma}
\begin{proof}
Fix a bijection $\pi:\mathbb N \to \mathbb N$ and let us show that the sequence $(\Pi_{i=1}^ng_{\pi(i)})_{n\in \mathbb N}$ converges to $g$.
\par
Fix $U\in \mathcal N(G)$ and let us find $n_U\in \mathbb N$ such that
\begin{equation}\label{kreq01}
 n\in\mathbb N,\,n\ge n_U\,\Longrightarrow\,\prod_{i=1}^ng_{\pi(i)}\in gU\,.
 \end{equation}
 Since $(g_i)_{i\in \mathbb N}$ is multipliable to $g$, there is 
$\alpha_U\in  \mathfrak F(\mathbb N)$ such that 
\begin{equation}\label{kreg02}
\alpha\in  \mathfrak F(\mathbb N),\,\,\alpha\supset \alpha_U\,\Longrightarrow\, \prod_{i\in \alpha}g_i\in gU\,.
\end{equation}
Let $n_U:=\max \pi^{-1}(\alpha_U)$. Take an arbitrary $n\in \mathbb N$ with $n\ge n_U$. Since $\pi$ is a bijection, we have:
 \begin{equation}\label{kreg03}
 \pi(\{1,\dots,n\})\supset \pi(\{1,\dots,n_U\})\supset \alpha_U\,. 
 \end{equation}
 From (\ref{kreg03}) and (\ref{kreg02}) we get:
 \begin{equation}\label{kreg04}
  \prod_{i\in \pi(\{1,\dots,n\}) }g_i\in gU\,.
 \end{equation}
 Clearly,
 \begin{equation}\label{kreg05}
  \prod_{i\in \pi(\{1,\dots,n\}) }g_i=\prod_{i=1}^ng_{\pi(i)}\,.
 \end{equation}
 Now (\ref{kreg04}) and (\ref{kreg05}) imply:
 \begin{equation}\label{kreg06}
  \prod_{i=1}^ng_{\pi(i)}\in gU\,.
 \end{equation}
 Therefore for an arbitrary $U\in \mathcal N(G)$ and we have found  $n_U\in \mathbb N$ for which (\ref{kreq01}) is satisfied. Hence, the sequence $(\Pi_{i=1}^ng_{\pi(i)})_{n\in \mathbb N}$ converges to $g$.  
\end{proof}
\begin{remark}{\em
The following converse to Lemma \ref{alo2} is true:\\
 {\it Let $(g_i)_{i\in \mathbb N}$ be a sequence of elements of $G$ such that for every bijection 
$\pi:\mathbb N \to \mathbb N$  the sequence $(\Pi_{i=1}^ng_{\pi(i)})_{n\in \mathbb N}$ converges, then
 $(g_i)_{i\in \mathbb N}$ is multipliable in $G$} \cite[Ch.III,\S5.7, Proposition 9]{Bou1}.\\
  This statement will not be used below.

}

\end{remark}
A family $(g_i)_{i\in I}$ of elements of $G$ will be called {\it pre-multipliable}  in $G$ if for every $U\in \mathcal N(G)$ there is $\beta_U\in \mathfrak F(\mathbb N)$ such that 
$$
\beta  \in \mathfrak F(I)\,\,,\,\,\beta\cap \beta_U=\emptyset\,\Longrightarrow\, \prod_{i\in \beta}g_i\in U\,.
$$

\begin{lemma}\label{alo4}
Let $(g_i)_{i\in \mathbb N}$ be a pre-multipliable sequence of elements of $G$, $g\in G$. If the sequence  $(\Pi_{i=1}^ng_{i})_{n\in \mathbb N}$ converges to $g$, then $(g_i)_{i\in \mathbb N}$ is multipliable to $g$.
\end{lemma}
\begin{proof}
Fix $U\in \mathcal N(G)$, find a $V\in \mathcal N(G)$ such that $VV\subset U$.
\par
 Since the sequence  $(\Pi_{i=1}^ng_{i})_{n\in \mathbb N}$ converges to $g$, there is $n_V\in \mathbb N$ such that
 \begin{equation}\label{kreq}
 n\in\mathbb N,\,n\ge n_V\,\Longrightarrow\,\prod_{i=1}^ng_i\in gV\,.
 \end{equation}
 Since  $(g_i)_{i\in \mathbb N}$ is pre-multipliable, there is $\beta_V\in \mathfrak F(\mathbb N)$ such that 
\begin{equation}\label{kreq2}
\beta  \in \mathfrak F(\mathbb N)\,\,,\,\,\beta\cap \beta_V=\emptyset\,\Longrightarrow\, \prod_{i\in \beta}g_i\in V\,.
\end{equation}
Let $N_V:=n_V+\max \beta_V$ and $\alpha_U:=\{1,\dots,N_V\}$ (we agree: $\max\emptyset:=0$). Fix now an arbitrary $\alpha \in \mathfrak F(\mathbb N)$ with $\alpha\supset \alpha_U$ and write: $\beta:= \alpha\setminus \alpha_U$. Observe that
\begin{equation}\label{kreq3}
\prod_{i\in \alpha}g_i=\left(\prod_{i=1}^{N_V}g_i\right)\left(\prod_{i\in \beta}g_i\right)\,. 
\end{equation}
Since $N_V\ge n_V$, from (\ref{kreq}) we conclude:
\begin{equation}\label{kreq4}
\prod_{i=1}^{N_V}g_i\in gV\,.
\end{equation}
Since $\alpha\cap \beta_V=\emptyset$, the relation (\ref{kreq2}) implies:
\begin{equation}\label{kreq5}
\prod_{i\in \beta}g_i\in V\,.
\end{equation}
Now (\ref{kreq4}) and (\ref{kreq5}) together with (\ref{kreq3}) imply:
\begin{equation}\label{kreq6}
\prod_{i\in \alpha}g_i=\left(\prod_{i=1}^{N_V}g_i\right)\left(\prod_{i\in \beta}g_i\right)\in gVV\subset gU\,. 
\end{equation}
Consequently for an arbitrary $U\in \mathcal N(G)$ we have found a $\alpha_U \in \mathfrak F(\mathbb N)$ such that (\ref{kreq6}) is satisfied for every 
$\alpha\in  \mathfrak F(\mathbb N)$ with $\alpha\supset \alpha_U$. Therefore,   $(g_i)_{i\in \mathbb N}$ is multipliable to $g$.
\end{proof}
\begin{lemma}\label{alo5}
Let $(g_i)_{i\in \mathbb N}$ be a super-multipliable sequence of elements of $G$. Then  $(g_i)_{i\in \mathbb N}$ is pre-multipliable.
\end{lemma}
\begin{proof}
Let us assume that $(g_i)_{i\in \mathbb N}$ is not pre-multipliable and derive from this that then $(g_i)_{i\in \mathbb N}$ is not  a super-multipliable sequence.
\par
Since $(g_i)_{i\in \mathbb N}$ is not pre-multipliable, we can find and fix a $U\in \mathcal N(G)$ such that 
\begin{equation}\label{aloeq1}
\forall n\in \mathbb N\,\,\exists \beta  \in \mathfrak F(\mathbb N),\,\,\,\,\beta\cap \{1,\dots,n\}=\emptyset,\,\,\,\,\, \prod_{i\in \beta}g_i\not \in U\,.
\end{equation}
Using (\ref{aloeq1}) we can construct a sequence $(\beta_n)_{n\in \mathbb N}$ of elements of $\mathfrak F(\mathbb N)$ such that
\begin{equation}\label{aloeq2}
1<\min \beta_n,\, \max\beta_n<\min \beta_{n+1},\,\,\prod_{i\in \beta_n}g_i\not \in U,\,\,\,n=1,2,\dots
\end{equation}
Define now a sequence $m:\mathbb N\to \{0,1\}$ as follows:\\
(1) If $i\in \mathbb N$ is such that $i\in \beta_n$ for some $n\in \mathbb N$, then $m_i=1$.\\
(2) If $i\in \mathbb N\setminus \cup_{n\in \mathbb N}\beta_n$, then $m_i=0$.
\par
Clearly (\ref{aloeq2}) implies:
\begin{equation}\label{aloeq3}
\prod _{i=\min \beta_n}^{\max \beta_n}g_i^{m_i}=\prod_{i\in \beta_n}g_i\not \in U,\,\,\,n=1,2,\dots
\end{equation}
From (\ref{aloeq3}), as $\min\beta_n\to \infty$, we conclude that $(\prod _{i=1}^{\max \beta_n}g_i^{m_i})_{n\in \mathbb N}$ is not a Cauchy sequence. Hence, 
$(\prod _{i=1}^{\max \beta_n}g_i^{m_i})_{n\in \mathbb N}$ is not a converging sequence. Consequently, 
$$\left(\prod _{i=1}^{n}g_i^{m_i}\right)_{n\in \mathbb N}$$
 is not a converging sequence either.
\par
Therefore we have found a sequence $m:\mathbb N\to \{0,1\}$ such that $(\prod _{i=1}^{n}g_i^{m_i})_{n\in \mathbb N}$ is not a converging sequence. This means that 
$(g_i)_{i\in \mathbb N}$ is not  a super-multipliable sequence.
\end{proof}
\begin{remark}{\em The method of proof of Lemma \ref{alo5} is seemingly due to \cite{hild}.  
}
\end{remark}
{\it Proof of Lemma \ref{vigo2}$(a)$.}\\
Since $(g_i)_{i\in \mathbb N}$ is a super-multipliable sequence of elements of $G$, the sequence  $(\Pi_{i=1}^ng_{i})_{n\in \mathbb N}$ converges to to some  $g\in G$. By Lemma \ref{alo5},  $(g_i)_{i\in \mathbb N}$  is pre-multipliable. By Lemma \ref{alo4} $(g_i)_{i\in \mathbb N}$ is multipliable to $g$. Hence, by  Lemma \ref{alo2} for every bijection $\pi:\mathbb N\to \mathbb N$ the sequence $(\Pi_{i=1}^ng_{\pi(i)})_{n\in \mathbb N}$ converges to $g$. $\square$
\par

 {\bf Acknowledgements.} 
 We thank to professor Dmitri Shakhmatov for sending \cite{SS} to us  and to professor Dikran Dikranjan for pointing out a gap in the proof of Theorem \ref{shakh3} in the initial version  of this paper.
 \par
   The paper was written
while the second author was visiting   (August and September, 2009)  Departamento de M\'etodos Matem\'aticos y de Representaci\'on  (University of A Coru\~na, Spain). He is grateful for the invitation and good conditions of work.

\vspace{.5cm}

\noindent X. Dominguez,\\
 Dept. de M{\'e}todos Matem{\'a}ticos y de Representaci{\'o}n \\
 Universidad de A Coru{\~n}a, Campus de Elvi{\~n}a.
E-15071 A
Coru{\~n}a, Spain\\
e-mail: xdominguez@udc.es\\
\\
V. Tarieladze,\\
Muskhelishvili Institute of Computational Mathematics. \\
 0171, Tbilisi,  Georgia\\
e-mail: vajatarieladze@yahoo.com\\


\begin{thebibliography}{MMM}

\bibitem{BPR}
Bessaga, C.; Pelczynski, A.; Rolewicz, S.
{\it Some properties of the space $(s)$. }
Colloq. Math. 7 (1959), 45--51. 

\bibitem{Bou1}
Bourbaki, N.
{\it General Topology,} Part 1, Hermann, Paris, 1966.
\bibitem{Bou2}
Bourbaki, N.
{\it General Topology,} Part 2, Hermann, Paris, 1966.

\bibitem{Const1}
Constantinescu, C.
{\it Familles multipliables dans les groupes topologiques séparés.} (English summary) 
C. R. Acad. Sci. Paris Sér. A-B 282 (1976), no. 4, Ai, A191--A193. 




\bibitem{DSS}
Dikranjan, D.; Shakhmatov, D.; Spevak, J. {\it NSS and TAP properties in topological groups close to being compact}, arXiv:0909.2381v1 [Math.GN]12 Sep 2009.

\bibitem{tapstap}
Domínguez, X.; Tarieladze, V.
{\it Metrizable TAP and STAP groups,} arXiv:0909.1400v1 [Math.GN]8 Sep 2009.
\bibitem{dtslender}
Domínguez, X.; Tarieladze, V.
{\it Topological groups in which all hyper converging sequences are hyper multipliable,} in preparation.
 \bibitem{hild}
 Hildebrandt, T. H. {\it  On unconditional convergence in normed vector spaces.} Bull. Amer. Math. Soc. 46(1940), 959-962.

\bibitem{Jar}
Jarchow, Hans
{\it Locally convex spaces.} 
Mathematische Leitfäden. [Mathematical Textbooks] B. G. Teubner, Stuttgart, 1981. 548 pp. 






\bibitem{SS}
D. Shakhmatov, J. Sp\v ev\'ak, Group-valued continuous functions with the topology of pointwise convergence, Topology 
and its Applications (2009), doi:10.1016/j.topol.2009.06.022









\end{thebibliography}
\end{document}